\documentclass[12pt,twoside]{article}

\usepackage{amsmath}
\usepackage{amssymb}
\usepackage{amsthm}
\usepackage{enumerate}
\usepackage{color}
\usepackage[english]{babel}
\usepackage{tikz}
\usepackage{soul}

\numberwithin{equation}{section}
 
\usetikzlibrary{patterns}

\pgfdeclarepatternformonly{my crosshatch dots}{\pgfqpoint{-1pt}{-1pt}}{\pgfqpoint{1pt}{1pt}}{\pgfqpoint{15pt}{15pt}}
{
    \pgfpathcircle{\pgfqpoint{0pt}{0pt}}{1pt}
    \pgfpathcircle{\pgfqpoint{5pt}{5pt}}{2pt}
    \pgfusepath{fill}
}

\pgfdeclarepatternformonly{my vertical lines}{\pgfpointorigin}
{\pgfqpoint{8pt}{8pt}}{\pgfqpoint{9pt}{9pt}}%
{
  \pgfsetlinewidth{1pt}
  \pgfpathmoveto{\pgfqpoint{0.5pt}{0pt}}
  \pgfpathlineto{\pgfqpoint{0.5pt}{100pt}}
  \pgfusepath{stroke}
}

\pgfdeclarepatternformonly{my horizontal lines}{\pgfpointorigin}
{\pgfqpoint{7pt}{7pt}}{\pgfqpoint{9pt}{9pt}}%
{
  \pgfsetlinewidth{1pt}
  \pgfpathmoveto{\pgfqpoint{0pt}{0.5pt}}
  \pgfpathlineto{\pgfqpoint{100pt}{0.5pt}}
  \pgfusepath{stroke}
}

\pgfdeclarepatternformonly{my north west lines}{\pgfqpoint{-1pt}{-1pt}}{\pgfqpoint{4pt}{4pt}}{\pgfqpoint{3pt}{3pt}}%
{
  \pgfsetlinewidth{0.4pt}
  \pgfpathmoveto{\pgfqpoint{0pt}{3pt}}
  \pgfpathlineto{\pgfqpoint{3.1pt}{-0.1pt}}
  \pgfusepath{stroke}
}

\pgfdeclarepatternformonly{my north east lines}{\pgfqpoint{-2pt}{-2pt}}{\pgfqpoint{4pt}{4pt}}{\pgfqpoint{7pt}{7pt}}
{
  \pgfsetlinewidth{.5pt}
  \pgfpathmoveto{\pgfqpoint{0pt}{0pt}}
  \pgfpathlineto{\pgfqpoint{7pt}{7pt}}
  \pgfusepath{stroke}
}

\pgfdeclarepatternformonly{my fivepointed stars}{\pgfpointorigin}{\pgfqpoint{3mm}{3mm}}{\pgfqpoint{5mm}{5mm}}%
{
  \pgftransformshift{\pgfqpoint{1mm}{1mm}}
  \pgfpathmoveto{\pgfqpointpolar{25}{1mm}}
  \pgfpathlineto{\pgfqpointpolar{162}{1mm}}
  \pgfpathlineto{\pgfqpointpolar{306}{1mm}}
  \pgfpathlineto{\pgfqpointpolar{90}{1mm}}
  \pgfpathlineto{\pgfqpointpolar{234}{1mm}}
  \pgfpathclose%
  \pgfusepath{fill}
}

\newtheorem{theorem}{Theorem}[section]

\newtheorem{lemma}[theorem]{Lemma}

\newtheorem{remark}[theorem]{Remark}
\newtheorem{example}[theorem]{Example}
\newtheorem{definition}[theorem]{Definition}

\newcommand{\C}{{\ensuremath{\mathbb{C}}}}

\newcommand{\ind}{{\ensuremath{\rm Ind}}}

\newcommand{\rk}{{\ensuremath{\rm rk}}}
\newcommand{\rank}{{\ensuremath{\rm rk}}}

\newcommand{\odagger}{\mathrel{\text{\textcircled{$\dagger$}}}}

\usepackage[left=3cm,top=3.5cm,right=3cm,bottom=3.5cm]{geometry}

\begin{document}

\title{Simultaneous extension of generalized BT-inverses and core-EP inverses}

\author{Abdessalam Kara\thanks{ Faculty of Mathematics and
  Computer Sciences, Department of Mathematics, University of Batna 2, Batna, Algeria., E-mail: \texttt{abdessalam.kara@univ-batna.dz}.},
  N\'estor Thome\thanks{Instituto Universitario de Matem\'atica Multidisciplinar, Universitat Polit\`ecnica de
Val\`encia, 46022 Valencia, Spain. E-mail: \texttt{njthome@mat.upv.es}.},
Dragan Djordjević\thanks{University of Nis, Faculty of Sciences and Mathematics, Serbia. E-mail: \texttt{dragandjordjevic70@gmail.com}.}}
\maketitle

\begin{abstract}
In this paper we introduce the generalized inverse of complex square matrix with respect to other matrix having same size. Some of its representations, properties and characterizations are obtained. Also some new representation matrices of $W$-weighted BT-inverse and $W$-weighted core-EP inverse are determined as well as characterizations of generalized inverses $A^{\odagger}$$,A^{\odagger,W}$$, A^{\diamond}, A^{\diamond,W}$.
\end{abstract}
\maketitle
\makeatletter
\renewcommand\@makefnmark%
{\mbox{\textsuperscript{\normalfont\@thefnmark)}}}
\makeatother

\section{Introduction and preliminary}In this paper ${\mathbb C}^{m\times n}$ stands for the set of all complex matrices of size $m\times n$. The symbols $R(A), N(A), A^* \textit{ and }$$\rk(A)$ will denote the range space, the null space, the conjugate transpose and the rank of matrix $A\in \C^{m,n}$. Let $A\in {\mathbb C}^{m \times n}$, the Moore-Penrose inverse of $A$, denoted by $A^\dag$, is the unique matrix $X\in \C^{n\times m}$ satisfying the equations
$$ \textit{(i) } AXA=A, \textit{(ii) } XAX=X, \textit{(iii) }(AX)^*=AX, \textit{(iv) } (XA)^*=XA.$$
From this inverse we put $P_A=AA^\dag \textit{ and } Q_A= A^\dag A$. The class of matrices satisfying $(i)$ is denoted by $A\{1 \}$, also the class of matrices satisfying $(ii)$ is denoted by $A\{2 \}$. The index of a given square matrix $A$, denoted by $\ind(A)$, is the smallest nonnegative integer $k$ satisfying $\rk(A^{k+1})=\rk(A^{k})$. The Drazin inverse $A^D$ of square matrix $A\in \C^{n\times n}$ with $\ind(A)=k$ is defined as the unique solution of the equations
$$XAX=X,\textit{  } AX=XA,\textit{  } XA^{k+1}=A^{k}.$$
The core-EP inverse of square matrix $A\in \C^{n\times n}$ with $\ind(A)=k$ is given by $A^{\odagger}=A^k( (A^*)^k A^{k+1})^\dag (A^*)^k$ (see \cite{PrMo}), which is the unique solution of these equations
$$XAX=X,\textbf{  }(AX)^*=AX,\textit{  }XA^{k+1}=A^{k}.$$
For a given rectangular matrix  $W \in \C^{m\times n} $, the authors Ferreyra, Levis and Thome extended in \cite{FeLeTh1} the notion of core-EP inverse to the $W$- weighted core-EP inverse for rectangular matrix $ A \in \C^{m\times n}$, denoted by $A^{\odagger,W}$ and it is characterized as the unique solution of
$$ WAWX=P_{(WA)^k} \textit{,   } R(X)\subset R((AW)^k).$$
where $k= \max\{\ind(AW), \ind(WA)\}$. 
For a square matrix $A \in \C^{n\times n}$, the BT-generalized inverse of $A$ is defined by $A^{\diamond}=({AP_A})^\dag$ in \cite{BaTr2}, and extended by Ferreyra, Thome, Torigino (see \cite{FeThTo}) to new generalized inverse for rectangular matrix $A \in \C^{m\times n}$, called the $W$-weighted BT-inverse of $A$ where $W\in C^{m\times n}$ and it is given by $A^{\diamond,W}=({WAWAW(AW)^\dag})^\dag$. Moreover, it is the unique solution of the system of conditions
$$ AWX=AW{(WAWP_{AW})}^\dag \textit{,   } R(X)\subset R(P_{AW}(WAW)^*).$$
We need the following auxiliary lemmas:
\begin{lemma}\label{a}\cite[Theorem 4.1]{FeLeTh1}
    Let $B\in {\mathbb C}^{n \times m}$ be a nonzero matrix, $A\in {\mathbb C}^{m \times n}$ and $k= \max\{\ind(AB), \ind(BA)\}$. Then there exist two unitary matrices $U \in {\mathbb C}^{m \times m}$, $V \in {\mathbb C}^{n \times n}$, two nonsingular matrices $A_{1}, B_{1}\in {\mathbb C}^{t\times t}$ and two matrices $A_{2}\in {\mathbb C}^{m-t\times n-t}$ and $B_{2}\in {\mathbb C}^{n-t\times m-t}$ such that $A_{2}B_{2}$ and $B_{2}A_{2}$ are nilpotent of indices $\ind(AB)$ and $\ind(BA)$, respectively, with:
\begin{equation}\label{TCF1}
    A=U\left[\begin{array}{cc}
   A_{1}  & A_{12} \\
    0 & A_{2}
\end{array}\right]V^*
\textit{ and }
B=V\left[\begin{array}{cc}
   B_{1}  & B_{12} \\
    0 & B_{2}
\end{array}\right]U^*
\end{equation}
\end{lemma}
This is known as the core-EP decomposition of the pair $\{A,B\}$.
\begin{lemma} \cite{MoKo}\label{a2}
    Let $C=U\left[\begin{array}{cc}
   C_1  & C_2 \\
    0 & C_3
\end{array}\right]V^*$ such that $C_1\in {C^{t\times t}}$ is nonsingular and $U \in {\mathbb C}^{m \times m}$, $V \in {\mathbb C}^{n \times n}$ are unitary. Then
\begin{equation}\label{TCF2}
    C^\dag =V\left[ \begin{array}{cc}
        C^*_1\Omega &  -C^*_1\Omega C_2 C_3^\dag\\
     (I_{n-t} -Q_{C_3})C_2^*\Omega & C_3^\dag-(I_{n-t} -Q_{C_3})C_2^* \Omega C_2C_3^\dag
\end{array}\right]U^*
\end{equation}
and
\begin{equation}\label{TCF3}
   P_C= CC^\dag =U\left[ \begin{array}{cc}
        I_t & 0\\
     0 &  P_{C_3}
\end{array}\right]U^*
\end{equation}
where $\Omega= (C_1C^*_1+ C_2((I_{n-t} -Q_{C_3})C_2^*)^{-1}$.

\end{lemma}
This paper is organized as follows. In Section 2 we introduce the generalized inverse of square matrix with respect to other matrix having same size and its some representations. In Section 3 we give some properties of generalized inverse of square matrix with respect to other. Section 4 contains some characterizations of generalized inverse of square matrix with respect to other matrix. In Section 5 some new representation matrices of $W$-weighted BT-inverse and $W$-weighted core-EP inverses are obtained as well as other characterizations of $A^{\odagger},A^{\odagger,W}, A^{\diamond}, A^{\diamond,W}$.
\section{Definition of $A^{(B)}$ and some representations}
 We start this section with a theorem by giving a decomposition of square matrix by other of same size, for $A, B \in {\mathbb C}^{n \times n}$ we called $B$-decomposition of $A$ or decomposition of $A$ with respect to $B$.
\begin{theorem} \cite{Th}\label{a3}
    Let $A, B \in {\mathbb C}^{n \times n}$ with $\rk(A)=r$ and $\rk(B)=s$. Then there exist unitary matrices $U,V \in {\mathbb C}^{n \times n}$ and $A$, $B$ can be represented as follows
\begin{equation}\label{TCF4}
A= U\left[\begin{array}{cc}
\Sigma_A A_1 & \Sigma_A A_2 \\
0 & 0 
\end{array}
\right]V^* 
\qquad \text{ and } \qquad 
B= V\left[\begin{array}{cc}
\Sigma_B B_1\ &\Sigma_B B_2\\
 0  & 0 
\end{array}
\right]U^*
\end{equation}  
where $\Sigma_A \in \C^{r \times r}$ and $\Sigma_B\in \C^{s \times s}$ are diagonal positive definite matrices, blocks $A_1\in \C^{r \times s}$, $A_2\in \C^{r \times (n-s)}$, $B_1\in \C^{s \times r}$ and $B_2\in \C^{s\times (n-r) }$ satisfy $A_1 A_1^* + A_2 A_2^*=I_r$ and $B_1B_1^*  + B_2B_2^* =I_s$.\\

\end{theorem}
\begin{definition}
Let $A, B \in {\mathbb C}^{n \times n}$. The matrix $A^{(B)}=(ABB^\dag)^\dag$ is called the generalized inverse of $A$ with respect to $B$.
\end{definition}

\begin{remark}
If $B=A$, we recover the BT-generalized inverse since $A^{(A)}=(A^2A^\dag)^\dag=A^{\diamond}$. Also if $B=A^k$ with $k=\ind (A)$  we recover the core-EP inverse because $A^{(A^k)}= (AA^k(A^k)^\dag)^\dag= (A^{k+1}(A^k)^\dag)^\dag=A^{\odagger}$.
\end{remark}
Now we present a canonical form for the generalized inverse of $A$ with respect to $B$ by using the $B$-decomposition of $A$.
\begin{theorem}
Let $A, B \in {\mathbb C}^{n \times n}$ written as in (\ref{TCF4}) with $\rk(A)=r$. Then 
\begin{equation}\label{TCF6}
A^{(B)} = V\left[\begin{array}{cc}
(\Sigma_A A_1)^\dag & 0 \\
0 & 0 
\end{array}
\right] U^*.
\end{equation}
\end{theorem}
\begin{proof}

From (\ref{TCF4}) it follows
\[
B^\dag = U\left[\begin{array}{cc}
B_1^* \Sigma_B^{-1}  & 0 \\
B_2^* \Sigma_B^{-1}  & 0 
\end{array}
\right]V^*.
\]

Some computations yield
\begin{eqnarray*}
ABB^\dag & = &  U\left[\begin{array}{cc}
\Sigma_A A_1 & \Sigma_A A_2 \\
0 & 0 
\end{array}
\right] \left[\begin{array}{cc}
\Sigma_B B_1 & \Sigma_B B_2 \\
0 & 0 
\end{array}
\right]\left[\begin{array}{cc}
B_1^* \Sigma_B^{-1}  & 0 \\
B_2^* \Sigma_B^{-1}  & 0 
\end{array}
\right]V^* \\
 & = &  
U\left[\begin{array}{cc}
\Sigma_A A_1 & \Sigma_A A_2 \\
0 & 0 
\end{array}
\right] \left[\begin{array}{cc}
I_s   & 0 \\
0  & 0 
\end{array}
\right]V^* \\
& = & U\left[\begin{array}{cc}
\Sigma_A A_1 & 0 \\
0 & 0 
\end{array}
\right] V^*.
\end{eqnarray*}
Then,
$A^{(B)} = (ABB^\dag)^\dag= V\left[\begin{array}{cc}
(\Sigma_A A_1)^\dag & 0 \\
0 & 0 
\end{array}
\right] U^*$.
\end{proof}
We present a representation for generalized inverse of $A$ with respect to $B$ by using the core-EP decomposition of the pair $\{A,B\}$.
\begin{theorem}\label{TCF}
Let $A \in {\mathbb C}^{n \times n}$  and let $B \in {\mathbb C}^{n \times n}$ be a nonzero matrix as in (\ref{a}). Then
\begin{equation}\label{TCF7}
A^{(B)}=V\left[ \begin{array}{cc}
        A^*_{1}\Omega &  -A^*_{1}\Omega M N^\dag\\
     (I_{n-t} -Q_N)M^*\Omega & N^\dag-(I_{n-t} -Q_{N})M^* \Omega MN^\dag
\end{array}\right]U^*
\end{equation}
where $\Omega= (A_{1}A^*_{1}+ M((I_{n-t} -Q_{N})M^*)^{-1}$, $ M=A_{12}P_{B_{2}}, N=A_{2}P_{B_{2}}$
\end{theorem}

\begin{proof}
Since $B_{1}$  is nonsingular, it results from (\ref{TCF3}) that
\begin{center}
$P_B =U\left[ \begin{array}{cc}
        I_t & 0\\
     0 &  P_{B_{2}}
\end{array}\right]U^*.$
\end{center}
So 
\begin{center}
$ABB^\dag =V\left[ \begin{array}{cc}
       A_{1} & A_{12}P_{B_{2}}\\
     0 &  A_2 P_{B_{2}}
\end{array}\right]U^*.$
\end{center}
If we apply  (\ref{TCF2}) of Lemma \ref{a2} to $ABB^\dag$, we get the representation (\ref{TCF7}).
\end{proof}
\section{Some properties of $A^{(B)}$ }
For two given square matrices $A, B \in {\mathbb C}^{n \times n}$, in this section some properties of generalized inverse $A$ with respect to $B$ are obtained.
\begin{theorem}\label{a11}
   Let $A, B \in {\mathbb C}^{n \times n}$. Then
   \begin{enumerate}[a)]
   \item $R(A^{(B)})=R(BB^\dag A^*) \textit{ and } N(A^{(B)})=N((AB)^*),$
   \item $A^{(B)}\in A\{2\},$
   \item $A^{(B)}\in A\{1\}$ if and only if $\rk(AB)=\rk(A)$.
   \end{enumerate}
\end{theorem}
   \begin{proof}
      a) Clearly that $R(A^{(B)})=R(BB^\dag A^* )$. We have
    $N((AB)^*)=N(B^\dag A^*)\subset N(A^{(B)})=N(BB^\dag A^* )\subset N (B^\dag A^* )=N((AB)^*)$, so
    we get $N(A^{(B)})=N((AB)^*)$.\\ b) It follows from (\ref{TCF4}) and (\ref{TCF6}).\\c) Since $\dim R(A^{(B)}) =\dim R(ABB^\dag)= \dim R(AB)$ so by item b) and \cite[pg 46, Corollary 1]{BeGr} we deduce that the item c) holds. 
       
   \end{proof}

\begin{theorem} Let $A, B \in {\mathbb C}^{n \times n}$. Then
  \begin{enumerate} [a)]
      \item $AA^{(B)}$ is the orthogonal projector on $R(AB),$
      
      \item $A^{(B)}A$ is the oblique projector onto $R(BB^\dag A^*)$ along $N(B^*A^*A).$
  \end{enumerate}
\end{theorem}

\begin{proof}
From (\ref{TCF4}) and (\ref{TCF6}) we get
\begin{equation}\label{TCF8}
   AA^{(B)}= U\left[\begin{array}{cc}
(\Sigma_A A_1) (\Sigma_A A_1)^\dag& 0 \\
0 & 0 
\end{array}
\right] U^*.
\end{equation}
\begin{equation}\label{TCF9}
   A^{(B)}A= V\left[\begin{array}{cc}
(\Sigma_A A_1)^\dag \Sigma_A A_1& (\Sigma_A A_1)^\dag \Sigma_AA_2\\
0 & 0 
\end{array}
\right] V^*.
\end{equation}
Clearly that (\ref{TCF8}) yields to $AA^{(B)}$ is an orthogonal projector with $$R(AA^{(B)})= R(AB B^\dag(AB B^\dag)^\dag) = R(AB).$$Then results the item a). From the item b) of above theorem, $A^{(B)}A$ is projector, moreover $R(A^{(B)})=R(BB^\dag A^*)$ and $N(A^{(B)}A)=N((ABB^\dag)^\dag A)=[A^*R(ABB^\dag)]^\perp=[R(A^*AB)]^\perp=N(B^*A^*A)$, so the item b) holds.
\end{proof}

\begin{theorem}\label{a01}
 Let $A, B \in {\mathbb C}^{n \times n}$ as in (\ref{TCF4}). Then
 \begin{enumerate}[a)]
    \item $A^{(B)}=0$ if and only if $AB=0$,
    \item ${A^{(B)}}^\dag =ABB^\dag$,
    \item $BB^\dag{A^{(B)}}=A^{(B)}$,
    \item ${A^{(I)}}={A^{(A^\dag)}}=A^\dag$,
    \item ${A^{(A^{(B)})}} =A^{(B)}$,
    \item ${A^{(B)}}^{{[A^{(B)}]^\dagger}}={A^{(B)}}^{\dag}$,
    \item $AB=BA$ if and only if $[A^{(B)}]^\dag B=[B^{(A)}]^\dag A$, 
    \item $A^{(B)}=A^\dag$ if and only if $A_2=0$.
\end{enumerate}
\end{theorem}
\begin{proof}
    We have $$A^{(B)}=0 \iff (ABB^\dag)^\dag=0\iff ABB^\dag=0 \iff AB=0.$$ So the item a) holds. The items $b)$,  $c)$, $d)$, $e)$, $f)$ and $g)$ are easy to check by using (\ref{TCF4}) and (\ref{TCF6}). We prove the item $h)$:
    the equality $A^{(B)}=A^\dag$ means by (\ref{TCF4}) and (\ref{TCF6}) that $ABB^\dag=A $ which is
 \begin{equation} \label{TCF10}
U\left[\begin{array}{cc}
\Sigma_A A_1 & 0 \\
0 & 0 
\end{array}\right] V^*=
 U\left[\begin{array}{cc}
\Sigma_A A_1 &  \Sigma_A A_2\\
0& 0 
\end{array}\right] V^*
\end{equation}
This last is equivalent to $A_2=0$ because $\Sigma_A$ is nonsingular. 
\end{proof}
\begin{theorem} \label{v}
 Let $A \in {\mathbb C}^{n \times n}$ and let $ W \in {\mathbb C}^{n \times n}$ be a nonzero matrix. Then
 \begin{enumerate}[a)]
     \item For any $A, W \in {\mathbb C}^{n \times n}$, then $ A^{\diamond,W}=(WAW)^{(AW)}$, 
    \item For any $A, W \in {\mathbb C}^{n \times n}$, then $ A^{\odagger,W}=(WAW)^{((AW)^k)}$\\ where $k= \max\{\ind(AW), \ind(WA)\}$.
    
\end{enumerate}
\end{theorem}
\begin{proof}
    $a)$ In fact, for two given square matrices $A,W$, if we take $C=WAW$ and $F=AW$ we have that $C^{(F)}=(CFF^\dag)^\dag=(WAWAW (AW)^\dag)^\dag=A^{\diamond,W}$.\\
    Similarly we obtain the item $b)$.
\end{proof}

In \cite[Theorem 3.3]{FeThTo}, the $W$- weighted 
 BT-inverse of $A$ is given by $A^{\diamond,W}=({WAWAW(AW)^\dag})^\dag$ and in the \cite[Theorem 5.2]{FeLeTh1}, the $W$-weighted core-EP inverse of $A$ is given by $A^{\odagger,W}=({WAW(AW)^k((AW)^k)^\dag})^\dag$  accordingly to assertions $a)$ and $b)$ of Theorem \ref{v} $A^{\odagger,W}, A^{\diamond,W}$ are a particular cases of $A^{(B)}$, but the contrary not valid, the following example illustrate that $A^{(B)}$, $A^{\odagger,W}$ and $A^{\diamond,W}$ are not coincident
\begin{example}
 Let:
$A = \left[ \begin{array}{ccc}
1 & 0 & 0\\
0 & 1 & 2\\
0 & 0 & 0
\end{array}\right]
\qquad \text{ and } \qquad
B =  \left[ \begin{array}{ccc}
1 & 0 & 0\\
0 & 0 & 0\\
0 & 0 & 1
\end{array}\right]$\\
Then
\begin{center}
    $A^{(B)} = \left[\begin{array}{ccc}
1 & 0 & 0\\
0 & 0 & 0\\
0 & \frac{1}{2} & 1
\end{array}\right]$.
\end{center}
Let
$W=\left[\begin{array}{ccc}
    a_{11} & a_{12} & a_{13} \\
    a_{21} & a_{22} & a_{23}\\
    a_{31} & a_{32} & a_{33}
\end{array}\right]$, then we have\\
\begin{center}
$AW=\left[\begin{array}{cc}
S	&T\\
0&0	
\end{array}\right]
\qquad \text{ and } \qquad
(AW)^\dag=\left[\begin{array}{cc}
S^*(SS^*+TT^*)^\dag	&0\\
T^*(SS^*+TT^*)^\dag	&0	
\end{array}\right]$ 
\end{center}

\begin{center}
$AW(AW)^\dag=\left[\begin{array}{cc}
(SS^*+TT^*)(SS^*+TT^*)^\dag	&0\\
0	&0\\
\end{array}\right]$
\end{center}
where
\begin{center}
$S=\left[\begin{array}{cc}
a_{11} & a_{12}\\
a_{21}+2a_{31}	& a_{22}+2a_{32}	
\end{array}\right]
\qquad \text{ and } \qquad
T=\left[\begin{array}{c}
a_{13}\\ 
a_{23}+	2a_{33}
\end{array}\right]$.\\
\end{center}
It is clear that $WAWAW(AW)^\dag$ has the form 
\begin{center}
$WAWAW(AW)^\dag=\left[\begin{array}{cc}
M	& 0\\
N	& 0
\end{array}\right].$
\end{center}
This gives us
\begin{center}
$A^{\diamond,W}=({WAWAW(AW)^\dag})^\dag=\left[\begin{array}{cc}
M^\dag	& N^\dag\\
0	& 0
\end{array}\right].$
\end{center}
Let $k= \max\{\ind(AW), \ind(WA)\}$, Now from the above, the expression of $(AW)^{k}$ of the form
\begin{center}
 $(AW)^{k}= \left[\begin{array}{cc}
S^k & S^{k-1}T\\
0 & 0
\end{array}\right].$
\end{center}
In the same way we get
\begin{center}
$A^{\odagger,W}=({WAW(AW)^k((AW)^k)^\dag})^\dag=\left[\begin{array}{cc}
X^\dag	&Y^\dag\\
0	& 0\\
\end{array}\right].$
\end{center}
For same matrices $X$ and $Y$, finally note that for any $W \in {\mathbb C}^{3 \times 3}$
\begin{center}
$A^{(B)} = \left[\begin{array}{ccc}
1 & 0 & 0\\
0 & 0 & 0\\
0 & \frac{1}{2} & 1
\end{array}\right] \neq 
A^{\diamond,W}=\left[\begin{array}{cc}
M^\dag	&N^\dag\\
0	&0\\
\end{array}\right]$
\end{center}
And
\begin{center}
    $A^{(B)} = \left[\begin{array}{ccc}
1 & 0 & 0\\
0 & 0 & 0\\
0 & \frac{1}{2} & 1
\end{array}\right] 
\neq A^{\odagger,W}=\left[\begin{array}{cc}
X^\dag	&Y^\dag\\
0	&0\\
\end{array}\right].$
\end{center}
\end{example}

\section{Some characterizations of $A^{(B)}$ }
In this section, we present some characterizations of generalized inverse $A^{(B)}$ where the matrices $A, B$ are of same size. 
\begin{theorem}
Let $A, B \in {\mathbb C}^{n \times n}$. Then the matrix $A^{(B)}$ is characterized as the unique solution of the following system 
\begin{equation}\label{TCF11}
    XA=(ABB^\dag)^\dag A \textit{ and } R(X^*)\subset R(AB). 
\end{equation}
\end{theorem}
\begin{proof}
We have $A^{(B)}A = (ABB^\dag)^\dag A$ and $R({A^{(B)}}^*)=R(ABB^\dag)=R(AB)$, 
then $A^{(B)}$ satisfies (\ref{TCF11}). Now let $X$ be an other solution, we get
$A^{(B)}A=XA$ which means that $A^{*}({A^{(B)}}^{*}-X^*)=0$, consequently $R({A^{(B)}}^{*}-X^*)\subset N(A^{*})\cap R(AB)\subseteq R(A)^\bot\cap R(A)=\{0\}$
so $X=A^{(B)}$, thus $A^{(B)}$ is the unique solution of (\ref{TCF11}).
\end{proof}
\begin{theorem}
    Let $A, B \in {\mathbb C}^{n \times n}$. Then 
    \begin{equation}\label{a14}
    A^{(B)}=B(AB)^\dag
    \end{equation}
    and $A^{(B)}$ is characterized as the unique solution of the equations
   \begin{equation}\label{a16}
       XAX=X\textit{, } AX=P_{AB} \textit{, } XA=B{(AB)}^\dag A
   \end{equation}
\end{theorem}
\begin{proof}
    It is easy to check from (\ref{TCF4}) that
    \begin{equation}\label{TCF13}
        (AB)^{\dag}= U\left [\begin{array}{cc}
        B^*_1\Sigma^{-1}_B A^{\dag}_1\Sigma^{-1}_A & 0 \\
        B^*_2\Sigma^{-1}_B A^{\dag}_1\Sigma^{-1}_A & 0
        \end{array} \right]U^*
    \end{equation}
    thus, pre-multiplying (\ref{TCF13}) by $B$ we get $A^{(B)}=B(AB)^\dag$ and we conclude that $A^{(B)}$ is solution of (\ref{a16}), remaining prove that  
 the system (\ref{a16}) has unique solution. We suppose $X$ is other solution, we have
 $$X=XAX=XP_{AB} =XAA^{(B)}=B{(AB)}^\dag AA^{(B)}$$
 and from $A^{(B)}=B(AB)^\dag$ we get $X=A^{(B)} AA^{(B)}=A^{(B)}$, then it has unique solution.
\end{proof}
    Let $A\in \C^{n \times n}$ with $\rank(A)=r$ and let $ T \subset \C^n$ be a subspace of $\C^n$ with $\dim T=t \leq r$ and let $S$ be a subspace of $\C^n$ of dimension $n-t$, then there exists a unique generalized inverse  $X$ of $A$ satisfying $XAX= X$ having the range $T$ and the null space $S$  denoted by $A^{(2)}_{T,S}$ if and only if $A(T)\oplus S=\C^{n \times n}$. In the following theorem the generalized inverse of $A$ with respect to $B$ is characterized as the generalized inverse of $A$ having the range $R(BB^\dag A^* )$ and null space $N((AB)^*)$
    \begin{theorem}
        Let $A,B \in \C^{n \times n} $. Then
        \begin{equation}\label{TCF12}
            A^{(B)}=A^{(2)}_{R(BB^\dag A^* ),N((AB)^*)}
        \end{equation}
    \end{theorem}

\begin{proof}
We put $\rk(A)=r$, we have $\dim R(BB^\dag A^* )=\dim R(ABB^\dag)= \dim R(AB) \leq \dim R(A)=r$ and the rank theorem for matrices give us $\dim N((AB)^*)=n-\dim R(BB^\dag A^* )$. It follows from the Theorem \ref{a11} that $A^{(B)}$ is an outer inverse of $A$ having the range $R(BB^\dag A^* )$ and null space $N(AA^{(B)})=N(A^{(B)})=N((AB)^*)$, in addition 
    $$AR(BB^\dag A^* )\oplus N((AB)^*)=\C^{n \times n}=R(AB)\oplus N((AB)^*)=\C^{n \times n}$$
    so (\ref{TCF12}) holds.
    \end{proof}

\begin{theorem}
 Let $A,B \in \C^{n\times n} $. Then the following statements are equivalent
 \begin{enumerate}[a)]
     \item $X=A^{(B)}$,
     \item $X$ satisfies the equations
     $$XAX=X\textit{, } AX=A{(ABB^\dag)}^\dag \textit{, } XA={(ABB^\dag)}^\dag A ,$$
     \item $X$ satisfies the system
     $$AX=P_{AB}\textit{, } R(X)\subset R(BB^\dag A^*).$$
     
 \end{enumerate}
\end{theorem}
\begin{proof}
         a) implies b) It follows from the item b) of Theorem \ref{a11}.\\ 
         b) implies c) Using the equations $XAX=X\textit{, } XA={(ABB^\dag)}^\dag A$ it results that
         $$R(X)\subset R(XAX)\subset R(XA)\subset R(BB^\dag A^*)$$
         thus $R(X)\subset R(BB^\dag A^*)$ and by item c) of Theorem \ref{a01} and $R(ABB^\dag)=R(AB)$ we obtain that $AX=(ABB^\dag){(ABB^\dag)}^\dag=P_{AB}$.\\
         c) implies a) clearly that $X:=A^{(B)}$ satisfies the item c), we suppose that $X$ is a solution of system of c), note that $ R(X)\subset R(BB^\dag) , R(A^{(B)})\subset R(BB^\dag)$, so 
         $$A{(ABB^\dag)}^\dag=AX \implies ABB^\dag({(ABB^\dag)}^\dag-X)=0 $$
         Consequently, $R({(ABB^\dag)}^\dag-X) \subset N(ABB^\dag)\cap R(BB^\dag A^*)=\{0\}$ thus $X=A^{(B)}$.
     \end{proof}
     \section{ Some characterizations of $A^{\odagger},A^{\odagger,W}, A^{\diamond}, A^{\diamond,W}$}In this section, we give some new representation matrices of weighted-BT inverse and weighted core-EP inverse by using the decomposition of matrix with respect to other and some characterizations of $A^{\odagger},A^{\odagger,W}, A^{\diamond}, A^{\diamond,W}$ are obtained.
     \begin{theorem}
        Let $ A \in \C^{n \times n}$ with $\ind(A)=k$. Then the core-EP inverse of $A$ is the unique solution of this system
         \begin{equation}
            XAX=X, AX=AA^D(AA^D)^\dag, XA=A^D(AA^D)^\dag A
        \end{equation}   
      And $A^{\odagger}$ can be written follows
          \begin{equation}
             A^{\odagger}=A^{(A^D)}=A^D(AA^D)^\dag 
          \end{equation}
    \end{theorem}
    \begin{proof}
        It is easy to see that this system has unique solution, and $A^D(AA^D)^\dag$ satisfies it. On the one hand the applying of (\ref{a14}) {}and (\ref{TCF12}) yields to 
        \begin{center}
            $A^{(A^D)}=A^D(AA^D)^\dag=A^{(2)}_{R({A^D}{A^D}^*A^*),N(({AA^D})^*)}.$
        \end{center}
        On the other hand, we have $N(({AA^D})^{*})=N({A^{k}}^{*})$ and from
        \begin{center}
            $N({A^*}^k)=N({A^D}^*)\subset N(AA^D{A^D}^*) \subset N(A^D{A^D}^*)=N({A^D}^*)=N({A^*}^k).$
        \end{center}
        We derive that $R({A^D}{A^D}^*A^*)=R(A^k)$. 
        Thus $$A^{(A^D)}=A^D(AA^D)^\dag=A^{(2)}_{R(A^k),N({A^*}^k)}=A^{\odagger}. $$
    \end{proof}

The following theorem allows us to define the weighted core-EP inverse by only the index of $AW$ under suitable conditions for the matrices $A$ and $W$ without the need of index of $WA$.
\begin{theorem}
         Let $ W \in \C^{n \times m}$ be a nonzero matrix, $ A\subset \C^{m \times n}$ with $\ind(AW)=k$. Then
     \begin{equation}
             A^{\odagger, W}=(WAWP_{(AW)^k})^\dag.
    \end{equation}
    \end{theorem}
         \begin{proof} We consider ${k}^{'}=max \{\ind(AW), \ind(WA)\}$, since $R({(AW)}^{k^{'}})=R({(AW)}^{k})$, we get $P_{{(AW)}^{k}}=P_{{(AW)}^{k^{'}}}$, thus
             \begin{equation}
                 A^{\odagger, W}={(WAWP_{{(AW)}^{k^{'}}})}^{\dag}={(WAWP_{{(AW)}^k})}^{\dag}.
    \end{equation}
        \end{proof}
     \begin{theorem}
       Let $ A, W \in \C^{n \times n}$ be as in (\ref{TCF4}) where $W=B$ and $k=\ind(AW)$. Then $A^{\odagger,W}$ can be written as
         \begin{equation}\label{TCF01}
             A^{\odagger,W}=U\left[ \begin{array}{cc}
               (\Sigma_A A_1)^{\odagger,(\Sigma_W W_1)} \  & 0 \\
                 0 & 0
             \end{array} \right]V^*
         \end{equation}
     \end{theorem}

     \begin{proof}
     We have
     $$AW=U\left[ \begin{array}{cc}
               \Sigma_AA_1\Sigma_W W_1 \  & \Sigma_A A_1\Sigma_W W_2 \\
                 0 & 0
             \end{array} \right]U^*$$
             by \cite [Lemma 7.7.2]{CaMe}, $\ind{(\Sigma_W W_1\Sigma_A A_1)}=k-1$, and it is easy to see $(AW)^{\odagger}$ has the following matrix from
             $$(AW)^{\odagger}=U\left[ \begin{array}{cc}
               (\Sigma_AA_1\Sigma_W W_1)^{\odagger} \  & 0 \\
                 0 & 0
             \end{array} \right]U^*.$$
            We known that $(AW)^k((AW)^k)^\dag=(AW)(AW)^{\odagger}$, so we obtain
             $$ (AW)^k((AW)^k)^\dag =U\left[ \begin{array}{cc}
               (\Sigma_AA_1\Sigma_W W_1)(\Sigma_AA_1\Sigma_W W_1)^{\odagger} \  & 0 \\
                 0 & 0
             \end{array} \right]U^*.$$
            this means that
             \begin{equation}
                 (AW)^k((AW)^k)^\dag= U\left[ \begin{array}{cc}
               (\Sigma_AA_1\Sigma_W W_1)^k((\Sigma_AA_1\Sigma_W W_1)^k)^\dag &\ \\
                 0 & 0
             \end{array} \right]U^*.
             \end{equation}
             Thus
             \begin{equation}
                 A^{\odagger,W}= U\left[ \begin{array}{cc}
               [(\Sigma_W W_1\Sigma_A A_1\Sigma_W W_1) (\Sigma_AA_1\Sigma_W W_1)^k((\Sigma_AA_1\Sigma_W W_1)^k)^\dag ]^\dag&\ \\
                 0 & 0
             \end{array} \right]V^*
             \end{equation}
             which is
             $$A^{\odagger,W}=U\left[ \begin{array}{cc}
               (\Sigma_A A_1)^{\odagger, (\Sigma_W W_1)} & 0 \\
                 0 & 0
             \end{array} \right]V^*$$ 
             which is (\ref{TCF01}).
     \end{proof}
\begin{theorem} Let $A,W \in \C^{n\times n}$ with $k= \ind(AW)$. Then the following statements are equivalent
 \begin{enumerate}[a)]
    \item $X=A^{\odagger,W},$
     \item $X= A^{(2)}_{R((AW)^k[(AW)^k]^\dag(WAW)^*),N(WAW[(AW)^k]^*)},$
     \item $XWAW=(WAWP_{(AW)^k})^\dag)^\dag WAW \textit{ and } R(X^*)\subset R(W(AW)^{k+1}),$
     \item $XWAWX=X \textit{ , } WAWX=WAW(WAWP_{(AW)^k})^\dag \textit{ and } XWAW=X(WAWP_{(AW)^k})^\dag,$
     \item $WAWX=P_{W(AW)^{k+1}} \textit{ and } R(X)\subset R(P_{(AW)^k}(WAW)^*).$
 \end{enumerate}
\end{theorem}
     
    \begin{theorem} Let $ A \in \C^{n \times n}$. Then 
    $A^\diamond$  is the unique solution of the following equations
    \begin{equation}
        XAX=X \textit{,  } AX=P_{A^2} \textit{, } XA=A(A^2)^{\dag} A
    \end{equation}
     \end{theorem}
     \begin{proof}
         It follows from (\ref{a16}) when we replace $B$ by $A$.
     \end{proof}
\begin{theorem} Let $A, W \in \C^{n \times n}$ be as in (\ref{TCF4}) where $W=B$. Then $A^{\diamond,W}$ can be written as
         \begin{equation}\label{TCF0}
             A^{\diamond,W}=U\left[ \begin{array}{cc}
               (\Sigma_A A_1)^{\diamond,(\Sigma_W W_1)} \  & 0 \\
                 0 & 0
             \end{array} \right]V^*
         \end{equation}
     \end{theorem}

     \begin{proof}
          From the expression $A^{\diamond, W}=(WAWP_{AW})^\dag$ and (\ref{TCF4}), it follows (\ref{TCF0}).
     \end{proof}
        
\begin{theorem} Let $A,B \in \C^{n\times n}$. Then the following statements are equivalent
 \begin{enumerate}[a)]
    \item $X=A^{\diamond,W},$
     \item $X= A^{(2)}_{R(AW(AW)^\dag(WAW)^*),N((W^*(AW)^2)^*)},$
     \item $XWAW=(WAWP_{AW}^\dag)^\dag A\textit{ and } R(X^*)\subset R(W(AW)^2),$
     \item $XWAWX=X \textit{ , } WAWX=WAW(WAWP_{AW})^\dag \textit{ and } XWAW=X(WAWP_{AW})^\dag,$
     \item $WAWX=P_{W(AW)^2} \textit{ and } R(X)\subset R(P_{AW}(WAW)^*),$
     \item $X=XWAWX \textit{,  } WAWX=XP_{W(AW)^2} \textit{,  } XWAW=AW{(W(AW)^2)}^\dag WAW.$
 \end{enumerate}
\end{theorem}
\section{Acknowledgements}
The authors would like to thank the anonymous referees for their careful reading of the paper and their valuable comments and suggestions that help us to improve the reading of the paper.
\section{Funding}
The second author has been partially supported by Universidad Nacional de La Pampa, Facultad de Ingenier\'ia (Grant Resol. Nro. 135/19), Universidad Nacional del Sur, Argentina (Grant Resol. PGI 24/ZL22) and Ministerio de Econom\'ia, Industria y Competitividad (Spain) [Grant Red de Excelencia RED2022-134176-T].




\begin{thebibliography}{33}

\bibitem{BaTr} O.M. Baksalary, G. Trenkler, Core inverse of matrices, Linear. Multi. Algebra. \textbf{58} (6) (2010) 681-697.
\bibitem{BaTr2} O.M. Baksalary, G. Trenkler, On a generalized core inverse, Appl. Math. Comput. \text{236} (2014) 450-457.
\bibitem{BeGr} A. Ben-Israel, T.N.E. Greville, Generalized Inverses: Theory and Applications, second ed., Springer-Verlag, New York, 2003.
\bibitem{Bo} R. Bouldin, The pseudo-inverse of a product. SIAM J. Appl. Math. \textbf{24} (4)(1973) 489–495.
\bibitem{CaMe} S.L. Campbell, C.D. Meyer Jr, Generalized Inverse of Linear Transformation,. Dover, New York, Second Edition, 1991.
\bibitem{FeLeTh1} F.E. Ferreyra, F.E. Levis, N. Thome, Revisiting of the core EP inverse and its extension to rectangular matrices,  Quaest. Math. \textbf{41} (2) (2018) 265-281. 
\bibitem{FeThTo} D.E. Ferreyra, N. Thome, C. Torigino, The $W$-weighted BT inverse, Quaest. Math. \textbf{46} (2023) 359-374.
\bibitem{GaChPa} Y. Gao, J. Chen, P. Patricio, Representations and properties of the W-weighted core-EP inverse, Linear Multi. Algebra. \textbf{68} (2018) 1-15.
\bibitem{MaTh} S. Malik, N. Thome, On a new generalized inverse for matrices of an arbitrary index, Appl. Math. Comput. \textbf{226} (2014) 575-580.
\bibitem{MoKo} D. Mosic, M.Z. Kolundzija, Weighted CMP inverse of a operator between Hilbert spaces. Rev. R. Acad. Cienc. Exactas Fis. Nat. Ser. A Math. \textbf{113} (2019) 2155-2173.

\bibitem{PrMo} K.M. Prasad, K.S. Mohana, Core EP inverse, Linear. Multi. Algebra. \textbf{62} (2014) 792-802.
\bibitem{Th} N. Thome, A simultaneous canonical form of a pair of matrices and applications involving the weighted Moore–Penrose inverse, Appl. Math. Letters. \textbf{53} (2016) 112-118.
\bibitem{Wang} X. Wang, Core-EP decomposition and its applications, Linear Algebra Appl. \textbf{508} (2016) 289-300.

\bibitem{ZuCh} K.Z. Zue, Y.J Chen, The new Revisitation of core-EP inverse of Matrices, Filomat. \textbf{33} (2019) 3061–3072. 
\end{thebibliography}
\end{document}